\documentclass[12pt]{amsart}

\usepackage{amsmath,amssymb,amsthm,enumerate}

\newtheorem{theorem}{Theorem}[section]
\newtheorem{corollary}[theorem]{Corollary}
\newtheorem{lemma}[theorem]{Lemma}
\newtheorem{proposition}[theorem]{Proposition}

\theoremstyle{definition}

\newtheorem{example}[theorem]{Example}

\theoremstyle{remark}
\newtheorem{remark}[theorem]{Remark}

\newcommand{\bP}{\mathbb P}
\newcommand{\bC}{\mathbb C}

\newcommand{\bZ}{\mathbb Z}

\newcommand{\norm}{|\!|}

\begin{document}
\title[]{Local Holomorphic Isometries of a Modified Projective Space into a Standard Projective Space; Rational Conformal Factors}
\author{Peter Ebenfelt}
\address{Department of Mathematics, University of California at San Diego, La Jolla, CA 92093-0112}
\email{pebenfel@math.ucsd.edu}
\thanks{The author was supported in part by the NSF grant DMS-1301282.}
\begin{abstract}
We consider local modifications $\omega_n+f^*\omega_d$ of the Fubini-Study metric (with associated $(1,1)$-form  $\omega_n$) on an open subset $\Omega\subset \bC\bP^n$ induced by a local holomorphic mapping $f\colon \Omega\to \bP^d$. Our main result is that there are "gaps" in potential dimensions $m$ such that the modification can be obtained as $h^*\omega_m$ for some local holomorphic mapping $h\colon \Omega\to \bC\bP^m$. We also consider the case of rational conformal factors.
\end{abstract}


\maketitle

\section{Introduction}

We shall consider complex projective space $\bC\bP^n$ equipped with the standard Fubini-Study metric, and we shall denote by $\omega_n$ its associated $(1,1)$-form. In affine coordinates $z=(z_1,\ldots, z_n)$ in an affine chart $\cong \bC^n\subset \bC\bP^n$, we have
\begin{equation}\label{e:FB}
\omega_n=\frac{\sqrt{-1}}{2\pi}\partial \bar\partial \log \left(1+\sum_{i=1}^n|z_i|^2\right)=\frac{\sqrt{-1}}{2\pi}\partial \bar\partial \log \left(1+\norm z\norm^2\right).
\end{equation}
Let $\Omega\subset \bC\bP^n$ be an open subset and $F\colon \Omega\to \bC\bP^d$ a holomorphic mapping. For a nonnegative real number $\lambda$, the $\lambda$-modification of the Fubini-Study metric (in $\Omega$) induced by this mapping is the metric whose associated $(1,1)$-form is given by $\omega_{n,F,\lambda}=\omega_n+\lambda F^*\omega_p$. The considerations in this paper are local, so we shall assume that $\Omega$ and $F(\Omega)$ are contained in affine charts of $\bC\bP^n$ and $\bC\bP^d$, respectively; thus, if we express $F$ in affine coordinates, $F(z)=[1:f(z)]$ with $f(z)=(f_1(z),\ldots, f_n(z))$, then $\omega_{n,f,\lambda}=\omega_{n,F,\lambda}$ is given by
\begin{equation}\label{e:FBmod}
\omega_{n,f,\lambda}:=\omega_n+\lambda\frac{\sqrt{-1}}{2\pi}\partial \bar\partial \log \left(1+\sum_{i=1}^p|f_i(z)|^2\right)=\omega_n+\lambda\frac{\sqrt{-1}}{2\pi}\partial \bar\partial \log \left(1+\norm f\norm^2\right).
\end{equation}
We shall further assume that there is a positive integer $m$, a positive real number $\mu$, and a holomorphic mapping $h\colon \Omega\subset \bC^n\to \bC^m\subset \bC\bP^m$ such that $\omega_{n,f,\lambda}=\mu h^*\omega_m$; i.e.,
\begin{equation}\label{e:modeq}
\omega_n=\mu h^*\omega_m-\lambda f^*\omega_d.
\end{equation}
This situation, but in a more general setting where the source is a general simply connected K\"ahler manifold and the target is a product of projective spaces, was considered in the recent paper \cite{HuangY13} by X. Huang, and Y. Yuan. They show that strong rigidity properties hold under suitable number theoretic conditions on the conformal factors $\mu$ and $\lambda$. In the restrictive setting considered here, their result would state that if there are no positive rational numbers $s,t$ such that $s\lambda=t\mu$, then $f$ and $h$ extend as global holomorphic immersions $f\colon \bC\bP^n\to \bC\bP^d$, $h\colon \bC\bP^n\to \bC\bP^m$, and furthermore, $f$ and $h$ are both conformal isometries (i.e., $f^*\omega_d=a\omega_n$, $h^*\omega_m=b\omega_n$) with integral conformal factors $a,b$ such that $1=a\mu-b\lambda$. The reader is referred to \cite{HuangY13} for a discussion of the relevance and general context of this problem.

In this paper, we shall consider in some sense the opposite case, where the conformal factors $\lambda$ and $\mu$ are rational numbers, in which case the number theoretic condition in \cite{HuangY13} of course fails. In this case, the rigidity properties established in \cite{HuangY13} also fail, as is pointed out in that paper: The mappings $f$ and $h$ do not extend as global mappings and they are not conformal isometries, in general. However, there are range estimates that hold for the rank of $h$, in general depending on the dimension $d$ of the modification as well as the conformal factors $\mu$ and $\lambda$. In the special case where the conformal factor $\lambda$ equals one, there are "gaps" in the range of possible ranks of $h$ such that the integers in these gaps cannot occur as the rank of an $h$ satisfying \eqref{e:modeq} for any $f$ or integral $\mu$. (This phenomenon is akin to the codimensional gaps that are predicted by the Huang-Ji-Yin Gap Conjecture \cite{HuangJiYin09} for CR mappings between spheres. Indeed, the underlying reasons are similar, in both cases boiling down to rank properties of certain sums of squares; see Section \ref{s:SOS}.)

We shall say that a mapping $G\colon \Omega\to \bC\bP^N$ is {\it minimally embedded} if the image $G(\Omega)$ is not contained in a proper projective plane. Since projective space equipped with the Fubini-Study metric is a homogeneous space, there is no loss of generality in assuming that $0\in \Omega$ and that $f(0)=h(0)=0$. In this case, $g$ being minimally embedded is equivalent to the components of $g=(g_1,\ldots, g_N)$, in affine coordinates near $0$, being linearly independent. We first state our result in the special case where the conformal factors $\lambda$ and $\mu$ are both one:

\begin{theorem}\label{thm:main0} Let $\Omega\subset \bC^n\subset\bC\bP^n$ be a connected open set, $f\colon\Omega\to \bC^d\subset \bC\bP^d$ a minimally embedded holomorphic mapping, and $\omega_{n,f}=\omega_n+f^*\omega_d$ the $1$-modification of the Fubini-Study metric $\omega_n$ induced by $f$. Then, there is a minimally embedded holomorphic mapping $h\colon \Omega\to \bC^m\subset \bC\bP^m$, unique up to multiplication by a unitary $m\times m$ matrix, such that $\omega_{n,f}=h^*\omega_m$ or, equivalently,
\begin{equation}\label{e:hmod0}
\omega_{n}=h^*\omega_m-f^*\omega_d,
\end{equation}
and the dimension $m$ satisfies the following:
\smallskip

{\rm (i)} If $d\leq n$, then
\begin{equation}\label{e:main0bound}
n+\sum_{l=0}^{d-1} (n-l)=n(d+1)-\frac{d(d-1)}{2}\leq m\leq dn+n+d=n(d+1)+d.
\end{equation}

{\rm (ii)} If $d\geq n$, then $m\geq \max(n(n+3)/2,d)$.
\end{theorem}

\begin{remark}{\rm Observe that, for fixed $n$, the function $d\mapsto n(d+1)-\frac{d(d-1)}{2}$ is strictly increasing in $d$ for $1\leq d\leq n$.
}
\end{remark}

The existence of the mapping $h$ is trivial in this case (see the proof of Theorem \ref{thm:main0} below), and the uniqueness is a consequence of a well known lemma by D'Angelo \cite{D'Angelobook}. The main point of the theorem is the range estimates in (i) on the dimension $m$ for low dimensional ($d\leq n$) modifications. We note that there are "gaps" in the range of possible dimensions $m$ that can occur as target dimensions for $h$, {\it regardless} of the modification (i.e., regardless of the modifying mapping $f$ and dimension $d$). For instance, if $d=1$, then $2n\leq m\leq 2n+1$. If $d=2$, then $3n-1\leq m\leq 3n+2$, etc. As $d$ grows towards $n$, these possible ranges of $m$ will initially be disjoint, but the "gaps" between them will shrink until eventually (when $d\sim \sqrt{2n}$) they disappear. The gap intervals of dimensions $m$ for which {\it no} minimally embedded $h\colon \Omega\to \bC^m\subset \bC\bP^m$ is isometric to a $1$-modification induced by any $f$ go as follows (until they disappear):
\begin{equation}\label{e:1gaps}
(0,2n),\ (2n+1,3n-1),\ (3n+2,4n-3),\ \ldots,
\end{equation}
with the $d$th gap being given by $(n(d+1)+d,n(d+2)-d(d+1)/2)$, which as the reader can readily verify becomes empty when $d$ is sufficiently large, $d\sim \sqrt{2n}$ as mentioned above. The estimate provided in (ii) is of less interest. It will become very poor as the dimension $d$ of the modification grows; for generic choices of $f$ the dimension $m$ will grow on the order of the right hand side of \eqref{e:main0bound}. We formulate the gap result as a corollary:

\begin{corollary}\label{cor:maingap} Let $\Omega\subset \bC^n\subset\bC\bP^n$ be a connected open set, $h\colon\Omega\to \bC^m\subset \bC\bP^m$ a minimally embedded holomorphic mapping such that
\begin{equation}\label{e:gapd}
n(k+1)+k<m<n(k+2)-\frac{k(k+1)}{2},
\end{equation}
for some $k$. Then, $h^*\omega_m$ is not the $1$-modification $w_{n,f}=w_n+f^*\omega_d$ for any $f\colon\Omega\to \bC^d\subset \bC\bP^d$.
\end{corollary}

In order to formulate a result with general, rational conformal factors $\lambda$ and $\mu$, we need to introduce some terminology and notation. Let $\phi=(\phi_1,\ldots,\phi_a)$ and $\psi=(\psi_1,\ldots, \psi_b)$ be holomorphic mappings $\Omega\to \bC^a\subset \bC\bP^{a}$ and $\Omega\to \bC^b\subset \bC\bP^{b}$, respectively. The {\it tensor product} $\phi\otimes \psi$ is defined to be the mapping $\Omega\to \bC^{ab}\subset \bC\bP^{ab}$ whose components are $\phi_i\psi_j$ as $i$ and $j$ run over the sets $\{1,\ldots,a\}$ and $\{1,\ldots,b\}$, respectively, in some predetermined ordering of pairs $(i,j)$. The notation $\phi^{\otimes k}$ denotes the tensor product of $\phi$ with itself $k$ times. The {\it rank} of a holomorphic mapping $\phi\colon \Omega\to \bC^a\subset \bC\bP^{a}$ is the smallest integer $r$ such that the image $\phi(\Omega)$ is contained in an affine complex plane (or, equivalently, projective plane if considered as a mapping into $\bC\bP^{a}$) of dimension $r$; in particular, $\phi$ is minimally embedded in $\bC\bP^a$ if and only if the rank is $a$.

\begin{theorem}\label{thm:main1} Let $\Omega\subset \bC^n\subset\bC\bP^n$ be a connected open set, $f\colon\Omega\to \bC^d\subset \bC\bP^d$ a minimally embedded holomorphic mapping, and $a,b,c$ positive integers without a common prime factor. Let $\omega_{n,f, c/b}=\omega_n+(c/b) f^*\omega_d$ be the $c/b$-modification of the Fubini-Study metric $\omega_n$ induced by $f$. Assume that there is a minimally embedded holomorphic mapping $h\colon \Omega\to \bC^m\subset \bC\bP^m$ such that $\omega_{n,f,c/b}=(a/b)h^*\omega_m$ or, equivalently,
\begin{equation}\label{e:hmodrat}
\omega_{n}=\frac{a}{b}h^*\omega_m-\frac{c}{b}f^*\omega_d.
\end{equation}
If $(1,f)^{\otimes c}$ has rank $e+1$, then
\begin{equation}\label{e:ftensorrk}
cd\leq e\leq \sum_{k=1}^c\binom{d+k-1}{k}
\end{equation}
and the following hold:
\smallskip

{\rm (i)} If $e\leq n$ and $b=1$, then the following two inequalities hold:
\begin{equation}\label{e:mest1}
\left\{
\begin{aligned}
n(e+1)-\frac{e(e-1)}{2}\leq& \sum_{k=1}^a\binom{m+k-1}{k}\\
am\leq& n(e+1)+e.
\end{aligned}
\right.
\end{equation}

{\rm (ii)} If $e\geq n$ or $b\geq 2$, then
\begin{equation}\label{e:mest2}
\frac{n(n+3)}{2}\leq \sum_{k=1}^a\binom{m+k-1}{k}.
\end{equation}
\end{theorem}

We note that if $a=1$ in Theorem \ref{thm:main1}, then the estimates on $m$ are of the same type as those of Theorem \ref{thm:main0}, with the rank $e$ playing the role of the dimension $d$ in the latter theorem.
If $a\geq 2$, then the existence of the mapping $h$ has to be assumed, as it may not exist in general. We also note that in this case the sum on the right in the first inequality of \eqref{e:mest1} and on the right in \eqref{e:mest2} is a polynomial of degree $a$ in $m$ with positive coefficients. Consequently, the lower bound on $m$ provided by these two estimates will be roughly on the order of the $a$th root of the left hand sides. In the case (i) in Theorem \ref{thm:main1}, this means (as the reader can verify) that the possible intervals of $m$ provided by \eqref{e:mest1} will in general not be disjoint for different values of $e\geq cd$, as is the case in Theorem \ref{thm:main0}. Thus, the gaps in possible values of $m$, regardless of the modification, that exist when the conformal factors are both one, cannot be predicted (although they may still exist) for general rational conformal factors by Theorem \ref{thm:main1}, {\it except} for the first gap that exists for sufficiently small values of $a$ (compared to the dimension $n$): For fixed $c\geq 1$, we have $e\geq c$ and hence \eqref{e:mest1} and \eqref{e:mest2} imply that
$$
n(c+1)-\frac{c(c-1)}{2}\leq \sum_{k=1}^a\binom{m+k-1}{k}.
$$
(If $c\geq n$, we replace the left hand side by $n(n+3)/2$, but let us assume here that $c\leq n$.)
Observe that if $m=1$, then the right hand side equals $a$. It follows, regardless of the modification, that if $a<n(c+1)-c(c-1)/2$, then $m\geq 2$. Similarly, if $m=2$, then the right hand side equals $(a+1)(a+2)/2 -1$ and, hence, if $(a+1)(a+2)/2 -1<n(c+1)-c(c-1)/2$, then $m\geq 3$, etc. In general, Theorem \ref{thm:main1} should be regarded as estimates on $m$ for a given rank $e$, but estimates that do not depend on the modifying mapping $f$ itself. As above, the main point is the estimates for low ranks $e\leq n$. We should point out, however, that the gap phenomenon described in Corollary \ref{cor:maingap} holds for the possible {\it ranks} of $(1,h)^{\otimes a}$. This follows directly from the proofs of Theorems \ref{thm:main0} and \ref{thm:main1}.

\begin{remark}\label{rem:bad}{\rm If $a\geq 2$, then the lower bounds provided by Theorem \ref{thm:main1} are at best $m\geq n$, and in general worse than this. To see this, observe that in the case $a\geq 2$, if
\begin{equation}\label{e:bestbound}
\frac{n(n+3)}{2}\leq m+\frac{m(m+1)}{2},
\end{equation}
then the inequalities for lower bounds on $m$ in \eqref{e:mest1} and \eqref{e:mest2} hold. Thus, any lower bound $m\geq A$ that follows from Theorem \ref{thm:main1} is implied by the lower bound that follows from \eqref{e:bestbound}, and the reader can readily verify that this bound is precisely $m\geq n$.
}
\end{remark}

If $f$ in Theorem \ref{thm:main1} is assumed to be a rational mapping, then one can show (using Huang's Lemma; see below) that in fact $m\geq n$. Thus, in view of Remark \ref{rem:bad}, for rational mappings, the lower bounds in Theorem \ref{thm:main1} do not yield any new information beyond $m\geq n$. It is also possible, in this case, that estimates that are {\it linear} in $m$ could be proved, if further progress is made on the SOS problem (see Section \ref{s:SOS}) in the general situation. The case where $f$ is rational is discussed in Section \ref{s:polyf}.

Standard arguments will reduce the proofs of Theorems \ref{thm:main0} and \ref{thm:main1} to statements about ranks of certain sums of squares. (A reader unfamiliar with these standard arguments might want to read Section \ref{s:proofs} before reading Section \ref{s:SOS}.) The results concerning the latter that are needed for the proofs are stated and proved in Section \ref{s:SOS}. The proofs of Theorems \ref{thm:main0} and \ref{thm:main1} are then given in Section \ref{s:proofs}. A discussion of the case where the modifying map is rational is conducted in Section \ref{s:polyf}. Some examples are also given in the latter section.

The connection between results concerning sums of squares and isometric embedding problems has also been explored in, e.g., \cite{CatlinJPD97}, \cite{CatlinJPD99}.

\section{Ranks of Sums of Squares}\label{s:SOS}

In this section, we shall consider some Sums of Squares (SOS) problems that arise in the context of holomorphic mappings between projective spaces with (modified) Fubini-Study metrics. The (standard) connection will be made in Section \ref{s:proofs}.

\subsection{Setup and Basics} We shall first consider the following Sums of Squares (SOS) Equation in $z=(z_1,\ldots,z_n)\in \bC^n$:
\begin{equation}\label{e:SOS}
\left(\sum_{j=1}^n|z_j|^2\right)a(z,\bar z)=\sum_{k=1}^m|h_k(z)|^2,
\end{equation}
where $a(z,\bar z)$ is a real-analytic, Hermitian function; $h=(h_1,\ldots,h_m)$ is a local (germ at $0$ of a) holomorphic mapping $(\bC^n,0)\to (\bC^m,0)$ whose components are linearly independent over $\bC$ (or, equivalently, whose image is not contained in a proper subspace of $\bC^m$). For brevity, we shall also use the notation
$$
\norm z\norm^2:=\sum_{j=1}^n|z_j|^2,\quad \norm h\norm^2=\norm h(z)\norm^2:=\sum_{k=1}^m|h_k(z)|^2;
$$
the dimension of the complex space whose Euclidian norm is used will be clear from the context. Thus, equation \eqref{e:SOS} can then be written
\begin{equation}\label{e:SOS'}
\norm z\norm^2 a(z,\bar z)=\norm h\norm^2.
\end{equation}
If the Hermitian function $a(z,\bar z)$ is a polynomial, then it can then be written as a difference of finite squared norms:
\begin{equation}\label{e:DOS}
a(z,\bar z)=\sum_{i=1}^p|f_i(z)|^2-\sum_{j=1}^q|g_j(z)|^2=\norm f\norm^2-\norm g\norm^2,
\end{equation}
where $f=(f_1,\ldots, f_p)$, $g=(g_1,\ldots,g_q)$ are polynomial mappings. If $a$ is real-analytic but not polynomial, then a similar decomposition can be achieved with (in general, infinite dimensional) Hilbert space valued $f$ and $g$. In what follows, we shall assume that $a$ can be decomposed as a difference of finite squared norms (as in \eqref{e:DOS}), which is always the case if $a$ is polynomial. When the components of $f$ and $g$ are linearly independent (as can always be achieved), then the pair $(p,q)$ is called the {\it rank} of $a$. If $q=0$ (meaning that \eqref{e:DOS} can be achieved with $g\equiv 0$), then $a$ is said to be a (finite) SOS and we will simply refer to $p$ (rather than $(p,0)$) as its rank; thus, e.g., $\norm h\norm^2$ above is a finite SOS of rank $m$. A fundamental problem of general interest (and whose solution would have direct implications for the Huang-Ji-Yin Conjecture in CR geometry mentioned in the introduction) can be described as follows:

\medskip
\noindent
{\bf SOS Problem.} {\it Let $a(z,\bar z)$ be a Hermitian real-analytic function in neighborhood of $0$ in $\bC^n$
and assume that $\norm z\norm^2 a(z,\bar z)$ is a finite {\rm SOS}, i.e., there exists a holomorphic mapping $h=(h_1,\ldots, h_m)$
satisfying \eqref{e:SOS'}. Relate the possible values of the rank $m$ of the {\rm SOS} $\norm h\norm^2$ to the rank $(p,q)$ of $a$ and the dimension $n$.}
\medskip

The only general result known, to the best of the author's knowledge, about the SOS Problem is Huang's Lemma \cite{Huang99}, which states that if $a(z,\bar z)$ is not identically zero, then the rank $m\geq n$. In this paper, we shall only consider the SOS problem in the special case where $a(z,\bar z)$ itself is an SOS.

In what follows, we shall also use the following notation: Let $F=(F_1,\ldots,F_a)$ and $G=(G_1,\ldots, G_b)$ be local holomorphic mappings $(\bC^n, 0)\to \bC^a$ and $(\bC^n, 0)\to \bC^b$, respectively. Then, $F\oplus G$ denotes the mapping $(\bC^n,0)\to \bC^{a+b}$ given by $F\oplus G:=(F,G)$, and $F\otimes G$ the mapping $(\bC^n,0)\to \bC^{ab}$ whose components are $F_iG_j$ as $i$ and $j$ run over the sets $\{1,\ldots,a\}$ and $\{1,\ldots,b\}$, respectively, in some predetermined ordering of pairs $(i,j)$. We observe immediately that
\begin{equation}
\norm F\oplus G\norm^2=\norm F\norm^2+\norm G\norm^2;\quad \norm F\otimes G\norm^2=\norm F\norm^2\norm G\norm^2.
\end{equation}
We shall use the notation $V_F\subset \bC\{z\}\cong \mathcal O_n$ for the vector space over $\bC$ spanned by the components of $F$. Clearly, the rank of the SOS $\norm F\norm^2$ equals the dimension of $V_F$.

\subsection{The SOS problem when $a(z,\bar z)$ is an SOS} The following result concerning the case when $a$ is a bihomogeneous Hermitian polynomial SOS was proved in
\cite{GrHa13}
(Proposition 3) using an estimate by Macauley on the growth of the Hilbert function of a homogeneous polynomial ideal:

\begin{proposition}\label{prop:GrHa13}{\rm (\cite{GrHa13})}
Let $A(Z,\bar Z)$ be a bihomogeneous Hermitian polynomial in $Z=(Z_0, Z_1,\ldots,Z_n)$ and $\bar Z$, and assume that $A(Z,\bar Z)$ is an SOS of rank $p$, i.e.,
$$
A(Z,\bar Z)=\sum_{i=1}^p |F_i(Z)|^2,
$$
where $F_1(Z), \ldots, F_p(Z)$ are linearly independent homogeneous polynomials. If $p\leq n+1$, then the rank $R$ of the SOS $\norm Z\norm^2A(Z,\bar Z)$ satisfies
\begin{equation}\label{e:GrHa13}
\sum_{l=0}^{p-1} (n+1-l)=(n+1)p-\frac{p(p-1)}{2}\leq R\leq p(n+1),
\end{equation}
and if $p\geq n+1$, then $R\geq (n+1)(n+2)/2$.
\end{proposition}

\begin{remark}\label{rem:GrHasharp} {\rm We note that both the lower and upper bound in
can be achieved for each $p\leq n$. It is easy to see that the lower bound is achieved with, e.g., $F_i(Z)=Z_i$ for $i=1,\ldots, p$. The upper bound is achieved for "generic" choices of $F_i(Z)$.
}
\end{remark}

A straightforward argument using homogenization of polynomials yields the following:

\begin{theorem}\label{thm:ASOS} Let $a(z,\bar z)$ be a Hermitian real-analytic function near $0$ in $\bC^n$, and assume that $a(z,\bar z)$ is a finite SOS of rank $p$, i.e.,
\begin{equation}\label{e:aSOSdecomp}
a(z,\bar z)=\sum_{i=1}^p |f_i(z)|^2,
\end{equation}
where $f_1(z), \ldots, f_p(z)$ are linearly independent holomorphic functions near $0$. If $p\leq n$, then the rank $r$ of the SOS $\norm z\norm^2a(z,\bar z)$ satisfies
\begin{equation}\label{e:ASOSbound}
\sum_{l=0}^{p-1} (n-l)=np-\frac{p(p-1)}{2}\leq r\leq pn,
\end{equation}
and if $p\geq n$, then $r\geq \max(n(n+1)/2,p)$.
\end{theorem}

\begin{proof} Let us first assume that $a(z,\bar z)$ is a polynomial of bidegree $(d,d)$ and that \eqref{e:aSOSdecomp} can be achieved with linearly independent polynomials $f_i(z)$ of degree at most $d$. Let us introduce homogeneous coordinates $Z=(Z_0,Z_1,\ldots,Z_n)=(Z_0,\tilde Z)$ and define homogeneous polynomials of degree $d$ by
\begin{equation}\label{e:homoF}
F_i(Z):=Z_0^df_i(\tilde Z/Z_0),\quad i=1,\ldots, p.
\end{equation}
Clearly, the $F_i$ are linearly independent since the $f_i$ are.
It follows that the bihomogeneous Hermitian polynomial
$$
A(Z,\bar Z):=|Z_0|^{2d}a\left(\tilde Z/Z_0, \overline{\tilde Z/Z_0}\right)
$$
then satisfies
$$
A(Z,\bar Z)=\sum_{i=1}^p |F_i(Z)|^2,
$$
and has rank $p$. Let us first assume that $p\leq n$. By Proposition \ref{prop:GrHa13}, the rank $R$ of the SOS $\norm Z\norm^2 A(Z,\bar Z)$ satisfies $R\geq (n+1)p-p(p-1)/2$. Since
$$
\norm Z\norm^2 A(Z,\bar Z)=(|Z_0|^2+\norm\tilde Z\norm^2)A(Z,\bar Z)=
\sum_{i=1}^p|Z_0|^2|F_i(Z)|^2+\norm \tilde Z\norm^2 A(Z,\bar Z)
$$
and $\norm \tilde Z\norm^2 A(Z,\bar Z)$ is also an SOS, it is then clear that the rank $r$ of $\norm \tilde Z\norm^2 A(Z,\bar Z)$ must satisfy
\begin{equation}\label{e:rankbound}
r\geq (n+1)p-p(p-1)/2 - p=np-p(p-1)/2,
\end{equation}
which when $p=n$ reduces to
\begin{equation}\label{e:rankbound'}
r\geq (n+1)n-n(n-1)/2 - n= n(n+1)/2.
\end{equation}
In other words, there are linearly independent homogeneous polynomials $H_i(Z)$ of degree $d+1$ such that
\begin{equation}\label{e:prelim}
\norm \tilde Z\norm^2 A(Z,\bar Z)=\sum_{i=1}^r |H_i(Z)|^2,
\end{equation}
where $r$ satisfies the lower bound \eqref{e:rankbound}. Noting that $f_i(z)=F_i(1,z)$ and $a(z,\bar z)=A\big ((1,z),\overline{(1,z)}\big)$, we conclude by substituting $Z=(1,z)$ in \eqref{e:prelim} that
\begin{equation}\label{e:ASOSdecomp}
\norm z\norm^2 a(z,\bar z)=\sum_{i=1}^r |h_i(z)|^2,
\end{equation}
where $h_i(z)=H_i(1,z)$. The $h_i$ are linearly independent since the $H_i$ are, and therefore the rank of $\norm z\norm^2 a(z,\bar z)$ equals $r$, where $r$ satisfies \eqref{e:rankbound}. Clearly, by construction we have $r\leq pn$ (since $pn$ is the total number of terms obtained when the product $\norm z\norm^2 a(z,\bar z)$ is multiplied out),  proving \eqref{e:ASOSbound} when $p\leq n$. If $p\geq n$, then the rank $r$ will be greater than or equal to the corresponding rank obtained when $a(z,\bar z)$ is replaced by the sum on the right in \eqref{e:aSOSdecomp} truncated after $n$ terms, i.e., $r\geq n(n+1)/2$. The fact that $r\geq p$ is trivial. This establishes the statement of Theorem \ref{thm:ASOS} in the polynomial case.

Next, let $a(z,\bar z)$ be a Hermitian real-analytic function near $0$ in $\bC^n$ satisfying \eqref{e:aSOSdecomp}, where the $f_i$ are linearly independent holomorphic functions near $0$, and let
\begin{equation}\label{e:ASOSdecomp2}
\norm z\norm^2 a(z,\bar z)=\sum_{i=1}^r |h_i(z)|^2,
\end{equation}
be a SOS decomposition of $\norm z\norm^2 a(z,\bar z)$ with the $h_i(z)$ being linearly independent holomorphic functions near $0$. If we truncate the Taylor series of the $f_i(z)$ at degree $d$ and those of the $h_i(z)$ at degree $d+1$, then we obtain a Hermitian polynomial $a^d(z,\bar z)$ of bidegree $(d,d)$ such that
\begin{equation}\label{e:adSOSdecomp}
a^d(z,\bar z)=\sum_{i=1}^p|f^d(z)|^2,\quad \norm z\norm^2 a^d(z,\bar z)=\sum_{i=1}^p|h^{d+1}(z)|^2
\end{equation}
where the $f^d_i(z)$ and $h^{d+1}(z)$ denote the truncated Taylor polynomials of $f_i(z)$ and $h_i(z)$ at degrees $d$ and $d+1$, respectively. Since the sets of holomorpic functions $f_1,\ldots, f_p$ and $h_1,\ldots, h_r$ both are linearly independent, it is clear that the sets of polynomials $f^d_1,\ldots, f^d_p$ and $h^d_1,\ldots, h^d_r$ both are linearly independent for $d$ sufficiently large. In other words, for $d$ sufficiently large, the Hermitian polynomial $a^d(z,\bar z)$ is an SOS of rank $p$ and $\norm z\norm^2 a^d(z,\bar z)$ is an SOS of rank $r$. The conclusion of Theorem \ref{thm:ASOS} now follows from the corresponding statement in the polynomial case, already established above. This concludes the proof of Theorem \ref{thm:ASOS}.
\end{proof}

We are now ready to state and prove a "nonhomogeneous" version of Theorem \ref{thm:ASOS} that will be used in the proofs of Theorems \ref{thm:main0} and \ref{thm:main1} below.

\begin{theorem}\label{thm:nonhomoSOS} Let $f_1(z),\ldots,f_p(z)$ be linearly independent local holomorphic functions vanishing at $0$ in $\bC^n$ such that
\begin{equation}\label{e:nonhomoSOSdecomp}
(1+\norm z\norm^2)(1+\sum_{i=1}^p |f_i(z)|^2)=1+\sum_{i=1}^r|h_i(z)|^2,
\end{equation}
where $h_1(z), \ldots, h_r(z)$ are linearly independent local holomorphic functions near $0$. If $p\leq n$, then the rank $r$ of the SOS $\norm h\norm^2=\sum_{i=1}^r |h_i(z)|^2$ satisfies
\begin{equation}\label{e:nonhomoASOSbound}
n+\sum_{l=0}^{p-1} (n-l)=n(p+1)-\frac{p(p-1)}{2}\leq r\leq pn+n+p=n(p+1)+p,
\end{equation}
and if $p\geq n$, then $r\geq n(n+3)/2$.
\end{theorem}

\begin{remark}{\rm Both the upper and lower bound in \eqref{e:nonhomoASOSbound} can be achieved by examples similar to those in Remark \ref{rem:GrHasharp}.
}
\end{remark}

\begin{proof} If we write $f=(f_1,\ldots, f_p)$ and $h=(h_1,\ldots, h_r)$, then \eqref{e:nonhomoSOSdecomp} can be written
$$
(1+\norm z\norm^2)(1+\norm f\norm^2)=1+\norm h\norm^2,
$$
which when multiplied out is equivalent to
\begin{equation}\label{e:nonhomo1}
\norm z\norm^2 +\norm f\norm^2+ \norm z\norm^2 \norm f\norm^2=\norm h\norm^2.
\end{equation}
It is clear that the vector spaces $V_z$ and $V_{f\otimes z}$ only intersect at $0$ (since the Taylor series of the $z_lf_i(z)$ have no constant or linear terms by the assumption that $f_i(0)=0$) and therefore the rank of the SOS $\norm z\norm^2 + \norm z\norm^2 \norm f\norm^2$ ($=\dim_\bC V_z\oplus V_{f\otimes Z}$) is $\geq$ the rank of $\norm z\norm^2 \norm f\norm^2$ plus $n$. It then follows immediately from Theorem \ref{thm:ASOS} that if $p\leq n$, then
$$
r\geq np -\frac{p(p-1)}{2} +n =n(p+1)-\frac{p(p-1)}{2},
$$
which is the lower bound in \eqref{e:nonhomoASOSbound}. The lower bound $r\geq n(n+3)/2$ when $p\geq n$ follows from Theorem \ref{thm:ASOS} in the same way. The upper bound in \eqref{e:nonhomoASOSbound} is obtained directly by counting the number of squares on the left in \eqref{e:nonhomo1}. This completes the proof of Theorem \ref{thm:nonhomoSOS}.
\end{proof}

\subsection{Another Sums of Squares Problem} For the proof of Theorem \ref{thm:main1}, we shall also need the following result concerning the rank of powers of $(1+\norm f\norm^2)$:

\begin{proposition}\label{prop:powerSOS} Let $f_1(z),\ldots,f_p(z)$ be linearly independent local holomorphic functions vanishing at $0$ in $\bC^n$. Let $t\in \bZ_+$ and express $(1+\sum_{i=1}^p |f_i(z)|^2)^t$ as follows:
\begin{equation}\label{e:powerSOSdecomp}
(1+\sum_{i=1}^p |f_i(z)|^2)^t=1+\sum_{i=1}^r|h_i(z)|^2,
\end{equation}
where $h_1(z), \ldots, h_r(z)$ are linearly independent local holomorphic functions near $0$. Then the rank $r$ of the SOS $\norm h\norm^2=\sum_{i=1}^r |h_i(z)|^2$ satisfies
\begin{equation}\label{e:powerASOSbound}
tp\leq r\leq \sum_{k=1}^t\binom{p+k-1}{k}.
\end{equation}
\end{proposition}

\begin{remark} {\rm The lower bound in \eqref{e:powerASOSbound} is realized by taking $f_i(z)=z_1^i$, and the upper bound can be realized by choosing the $f_i$ to be "suitably spaced" monomials. For instance, to realize the upper bound if $p\leq n$, we can simply take $f_i(z)=z_i$; if $p>n$, then we can take the first $n$ $f_i$'s of this form, and then take subsequent ones to be monomials with an increment in the degrees so that the degrees of monomials up to order $t$ in $f_1,\ldots, f_k$ is lower than the degree of the monomials $f_{k+1},\ldots,f_p$.
}
\end{remark}

\begin{proof} Observe that the rank $r$ is the dimension of the complex vector space $V_F\subset \bC\{z\}$ spanned by the collection $F$ of all monomials $f^\alpha:=f_1^{\alpha_1}\ldots f_p^{\alpha_p}$ with $1\leq |\alpha|\leq t$. The upper bound in \eqref{e:powerASOSbound} is easily seen to hold. The number on the right in \eqref{e:powerASOSbound} is the number of distinct monomials of degree $\leq t$ in $p$ variables, and the rank $r$ can clearly not exceed this. To prove the lower bound, we proceed as follows. First, a moments reflection will convince the reader that, by iteratively replacing the $k$th generator $f_k(z)$ with a suitable linear combination $f_k(z)-\sum_{i=1}^{k-1}c_i f_i(z)$ if necessary, we may assume that the generators are of the form $f_i(z)=q_i(z)+O(|z|^{s_i+1})$ where the $q_i(z)$ are homogeneous polynomials of degree $s_i$ such that $q_1,\ldots, q_p$ are linearly independent. We shall need the following lemma, in which the notation $\bC[w]_{\leq t}$ is used for the space of polynomials in $w=(w_1,\ldots,w_q)$ of degree $\leq t$.

\begin{lemma}\label{lem:polyinj} For any positive integers $t$ and $n$, there are positive integers $a_1,\ldots a_n$ such that the algebra homomorphism $\bC[z]\to \bC[\zeta]$ induced by the map
$$p(z)\mapsto p(\zeta^{a_1},\ldots, \zeta^{a_n})
$$
is injective when restricted as a linear map
$\bC[z]_{\leq t}\to \bC[\zeta]_{\leq t\max(a_1,\ldots a_n)}$.
\end{lemma}

\begin{proof}[Proof of Lemma $\ref{lem:polyinj}$] It suffices to prove the lemma for $n=2$, since this result can then be repeated iteratively to "collapse" two variables to one until the desired map $\bC[z]_{\leq t}\to \bC[\zeta]_{\leq t\max(a_1,\ldots a_n)}$ is obtained. Thus, assume $n=2$. Choose $a_1\neq a_2$ to be prime numbers $\geq t+1$. Then, the induced homomorphism maps the monomial $z^\alpha:=z_1^{\alpha_1}z_2^{\alpha_2}$ to $\zeta^{a_1\alpha_1+a_2\alpha_2}$. Suppose, in order to reach a contradiction, that the linear map $\bC[z]_{\leq t}\to \bC[\zeta]_{\leq t\max(a_1,a_2)}$ is not injective. Then, there must be $|\alpha|\leq t$ and $|\beta|\leq t$ such that $z^\alpha$ and $z^\beta$ get mapped to the same monomial $\zeta^k$; i.e.,
$a_1\alpha_1+a_2\alpha_2=a_1\beta_1+a_2\beta_2$ or, equivalently, $a_1(\alpha_1-\beta_1)=a_2(\beta_2-\alpha_2)$. This is clearly a contradiction since $a_1\neq a_2$ are prime numbers $\geq t+1$ and $|\alpha_1-\beta_1|$, $|\beta_2-\alpha_2|$ are both $\leq t$. This completes the proof.
\end{proof}

We now return to the proof of Proposition \ref{prop:powerSOS}. By Lemma \ref{lem:polyinj}, we can find positive integers $a_1,\ldots, a_n$ such that the one-variable polynomials $\tilde f_i(\zeta)=f_i(\zeta^{a_1},\ldots, \zeta^{a_n})$, for $1=1,2,\ldots, p$, are linearly independent. The rank $r$ in Proposition \ref{prop:powerSOS} will be $\geq$ the dimension $\tilde r$ of the complex vector space $V_{\tilde F}\subset \bC\{z\}$ spanned by the collection $\tilde F$ of all monomials $\tilde f^\alpha:=\tilde f_1^{\alpha_1}\ldots \tilde f_p^{\alpha_p}$ with $1\leq |\alpha|\leq t$. Again, by iteratively replacing the $k$th generator $\tilde f_k(\zeta)$ with a suitable linear combination $\tilde f_k(\zeta)-\sum_{i=1}^{k-1}c_i \tilde f_i(\zeta)$ and then renumbering if necessary, we may assume that $\tilde f_i(\zeta)=b_i\zeta^{s_i}+O(\zeta^{s_1+1})$ where $b_i\neq 0$ and $1\leq s_1<\ldots<s_p$. As a final reduction, we note that the dimension $\tilde r$ is $\geq$ the dimension $r'$ of the complex vector space $V'\subset \bC\{z\}$ spanned by the collection $M$ of all monomials
$$\zeta^{\alpha_1s_1+\ldots +\alpha_ps_p},\quad 1\leq |\alpha|\leq t.
$$
To finish the proof, we shall show that $r'$ satisfies the lower bound in \eqref{e:powerASOSbound}. We shall prove this by induction on $p$. Clearly, if $p=1$, then $r'=t=tp$. Next, assume that the lower bound $r'\geq tp$ has been proved for $p<p_0$. Then, with $p=p_0$, we obtain at least $t(p-1)$ distinct monomials $\zeta^q$ with $q=\alpha_1s_1+\ldots+\alpha_{p-1}s_{p-1}$ such that $|\alpha|=|(\alpha_1,\ldots,\alpha_{p-1},0)|\leq t$. We must prove that we obtain at least $t$ new distinct monomials $\zeta^q$ with $q=\alpha_1s_1+\ldots+\alpha_{p-1}s_{p-1}+\alpha_ps_p$ where $\alpha_p\geq 1$ and $|\alpha|\leq t$. We claim that every $q=ks_{p-1}+(m-k)s_p$, for $k=0,\ldots, t-1$, is such that it cannot be obtained as $q=\alpha_1s_1+\ldots+\alpha_{p-1}s_{p-1}$ with $|\alpha|=|(\alpha_1,\ldots,\alpha_{p-1},0)|\leq m$. To see this, note that since $1\leq s_1<\ldots<s_p$ we have
$\alpha_1s_1+\ldots+\alpha_{p-1}s_{p-1}\leq ts_{p-1}$. Moreover, we have
$$
ks_{p-1}+(t-k)s_p - ts_{p-1} =(t-k)(s_p-s_{p-1})>0,\quad k=0,\ldots, t-1,
$$
which proves the claim, and shows that $r'\geq (t-1)p+t=tp$ also for $p=p_0$. This completes the proof of the proposition.
\end{proof}

\section{Proof of Theorems \ref{thm:main0} and \ref{thm:main1}}\label{s:proofs}

\begin{proof}[Proof of Theorem $\ref{thm:main0}$] Let $f\colon\Omega\to \bC^d\subset \bC\bP^d$  and $h\colon \Omega\to \bC^m\subset \bC\bP^m$ be minimally embedded holomorphic mappings satisfying \eqref{e:hmod0}. Since projective space equipped with the Fubini-Study metric is a homogeneous space, we may assume that $0\in \Omega$ and $f(0)=0$, $h(0)=0$. By \eqref{e:FB} and \eqref{e:FBmod}, we conclude that $\log (1+\norm z\norm^2)+\log(1+\norm f\norm^2)=\log (1+\norm h\norm^2)+\Phi$, where $\Phi$ is a real-valued polyharmonic function, i.e., $\Phi(z,\bar z)=\phi(z)+\overline{\phi(z)}$ near $0$ for some local holomorphic function $\phi(z)$. By comparing the Taylor series of the log-terms (no constant terms and no pure terms in $z$ or $\bar z$) with that of $\Phi$, we conclude that $\Phi\equiv 0$, and hence
$$
\log (1+\norm z\norm^2)+\log(1+\norm f\norm^2)=\log (1+\norm h\norm^2).
$$
By exponentiating this, we obtain the SOS identity
\begin{equation}\label{e:FBSOS1}
(1+\norm z\norm^2)(1+\norm f\norm^2)=(1+\norm h\norm^2).
\end{equation}
Conversely, if $f$ and $h$ satisfy this SOS identity, then \eqref{e:hmod0} holds. It follows that if $f$ is a minimally embedded holomorphic mapping, then there exists a mapping $h$ satisfying \eqref{e:FBSOS1} (simply carry out the multiplication on the left), and elementary linear algebra shows that we may assume (after replacing the $h_i$ obtained by multiplying out the left hand side of \eqref{e:FBSOS1} by a suitable basis for the vector space spanned by these) that $h$ is also minimally embedded. If there are two minimally embedded $h$ and $\tilde h$ that both satisfy \eqref{e:FBSOS1}, then $\norm h\norm^2=\norm\tilde h\norm^2$ and hence, by a lemma of D'Angelo \cite{D'Angelobook}, it follows that $h=U\tilde h$ for some unitary $m\times m$ matrix $U$. To finish the proof of Theorem \ref{thm:main0}, we recall that $f$ and $h$, with $f(0)=0$ and $h(0)=0$, are minimally embedded precisely when their components are linearly independent. Thus, the estimates in (i) and (ii) follow immediately from Theorem \ref{thm:nonhomoSOS}.
\end{proof}

\begin{proof}[Proof of Theorem $\ref{thm:main1}$] As in the proof of Theorem \ref{thm:main0} above, we may assume that $0\in \Omega$ and $f(0)=0$, $h(0)=0$. The same argument as in that proof also shows that the mappings $f$, $h$ satisfy
\begin{equation}\label{e:FBSOS2}
(1+\norm z\norm^2)^b(1+\norm f\norm^2)^c=(1+\norm h\norm^2)^a.
\end{equation}
As noted in Section \ref{s:SOS}, we have $(1+\norm f\norm^2)^c=\norm (1,f)^{\otimes c}\norm^2$. Thus, if the rank of $(1,f)^{\otimes c}$ is $e+1$, then we can find a minimally embedded holomorphic mapping $g\colon \Omega\to \bC^{e}$, with $g(0)=0$, such that
\begin{equation}
(1+\norm f\norm^2)^c=\norm (1,f)^{\otimes c}\norm^2=(1+\norm g\norm^2).
\end{equation}
The estimate \eqref{e:ftensorrk} for $e$ follows from Proposition \ref{prop:powerSOS}.
We may now rewrite \eqref{e:FBSOS2} as follows:
\begin{equation}\label{e:FBSOS3}
(1+\norm z\norm^2)^b(1+\norm g\norm^2)=(1+\norm h\norm^2)^a.
\end{equation}
The estimates in (i) and (ii) in the case $b=1$ now follow immediately by combining the estimates in Theorem \ref{thm:nonhomoSOS} and Proposition \ref{prop:powerSOS}. If $b\geq 2$, we observe that
$$
(1+\norm z\norm^2)^b(1+\norm g\norm^2)=(1+\norm z\norm^2)^{b-1}\left((1+\norm z\norm^2)(1+\norm g\norm^2)\right)
$$
and $(1+\norm z\norm^2)(1+\norm g\norm^2)$ has rank at least $n$ by Theorem \ref{thm:nonhomoSOS}. The estimate in (ii) now follows by again combining the estimates in Theorem \ref{thm:nonhomoSOS} and Proposition \ref{prop:powerSOS}. This completes the proof of Theorem \ref{thm:main1}.
\end{proof}

\section{The case where the mapping $f$ is rational; Examples}\label{s:polyf}

Let us examine closer the case where $f\colon \Omega\to \bC\bP^d$ is a rational mapping in Theorem \ref{thm:main1}. Clearly, by the uniqueness property of the mapping $h$ (up to linear transformations as explained in the proof of Theorem \ref{thm:main0}), it follows that $h\colon \Omega\to \bC\bP^m$ is also rational. Let us choose homogeneous coordinates $Z=[Z_0:Z_1:\ldots :Z_n]$ in $\bC\bP^n$. The Fubini-Study metric then has the associated $(1,1)$-form
\begin{equation}\label{e:FBhomo}
\omega_n=\frac{\sqrt{-1}}{2\pi}\partial \bar\partial \log \left(\sum_{i=0}^n|Z_i|^2\right)=\frac{\sqrt{-1}}{2\pi}\partial \bar\partial \log \norm Z\norm^2.
\end{equation}
We shall denote the rational mappings $\bC\bP^n\to \bC\bP^d$ and $\bC\bP^n\to \bC\bP^m$ corresponding to $f$ and $h$ by $F$ and $H$, respectively. Thus, in homogeneous coordinates on the target projective spaces, we have $F(Z)=[F_0(Z):\ldots: F_d(Z)]$ and $H(Z)=[H_0(Z):\ldots:H_m(Z)]$, where $F_i(Z)$ and $H_j(Z)$ are homogeneous polynomials, and the assumptions in Theorem \ref{thm:main1} are equivalent to the identity
\begin{equation}\label{e:polyid}
(\norm Z\norm^2)^b(\norm F\norm^2)^c=(\norm H\norm^2)^a.
\end{equation}
Let us proceed under the assumption that $a\geq 2$.
By complexifying \eqref{e:polyid}, i.e.\ replacing $\bar Z$ by an independent complex variable $\chi$ and using the notation $\bar \phi(\chi):=\overline{\phi(\bar\chi)}$, we obtain
\begin{equation}\label{e:polyidcplx0}
\left(\sum_{j=0}^n Z_j\chi_j\right)^b\left(\sum_{i=0}^dF_i(Z)\bar F_i(\chi)\right)^c=
\left(\sum_{k=0}^mH_k(Z)\bar H_k(\chi)\right)^a.
\end{equation}
Since the polynomial $\sum_{j=0}^n Z_j\chi_j$ is irreducible, the identity \eqref{e:polyidcplx0} implies that $\sum_{j=0}^n Z_j\chi_j$ divides $\sum_{k=0}^mH_k(Z)\bar H_k(\chi)$, and we conclude that there exists a (homogeneous) polynomial $R(Z,\chi)$ such that
\begin{equation}\label{e:polyidcplx1}
\left(\sum_{j=0}^n Z_j\chi_j\right)R(Z,\chi)=
\sum_{k=0}^mH_k(Z)\bar H_k(\chi),
\end{equation}
or
\begin{equation}\label{e:polyid2}
\norm Z\norm^2 R(Z,\bar Z)=
\norm H\norm^2.
\end{equation}
Substituting this in \eqref{e:polyid}, we obtain
\begin{equation}\label{e:polyid3}
\left(\norm Z\norm^2\right)^b\left(\norm F\norm^2 \right)^c=
\left(\norm Z\norm^2 \right)^a R(Z,\bar Z)^a.
\end{equation}
From \eqref{e:polyid3}, we can conclude that
the Hermitian polynomial $R(z,\bar z)$ belongs to various "positivity classes" introduced by D'Angelo and Varolin (see \cite{JPDVar04}; see also \cite{JPD11}). For instance, if $b\geq a$, then it follows that $R(Z,\bar Z)^a$ is an SOS; {\it however}, $R(z,\bar z)$ may not be an SOS itself in general. Thus, in order to use \eqref{e:polyid3} instead of \eqref{e:polyid} to estimate the dimension $m$, we would need to solve the SOS Problem in Section \ref{s:SOS} in its general form. Nevertheless, Huang's Lemma (described in Section \ref{s:SOS} above) applied to \eqref{e:polyid2} implies that $m\geq n$, which is a stronger lower bound in the first gap than that provided by Theorem \ref{thm:main1} for general, not necessarily rational mappings (see the discussion following the statement of Theorem \ref{thm:main1} and Remark \ref{rem:bad}). When $R(z,\bar z)$ is not an SOS, then the modification of $\omega_n$ obtained by applying $(\sqrt{-1}/2\pi)\bar\partial\partial$ to $\log$ of the left side of \eqref{e:polyid2} arises from the pullback of a mapping into a projective space with an indefinite Fubini-Study type "metric" (as in \cite{Calabi53}).

We should remark that for local holomorphic mappings $f$ and $h$, reducing the identity $$
(1+\norm z\norm^2)^b(1+\norm f\norm^2)^c=(1+\norm h\norm^2)^a
$$
to an identity of the form
\begin{equation}\label{e:notuseful}
(1+\norm z\norm^2)(1+Q(z,\bar z))=1+\norm h\norm^2
\end{equation}
is not useful, unless $Q(z,\bar z)$ itself is an SOS. Any mapping $h$ has a local analytic identity of this form with
$$
1+Q(z,\bar z)=\frac{1+\norm h\norm^2}{1+\norm z\norm^2}.
$$

We conclude this section by giving an example (taken from \cite{JPDVar04} and \cite{JPD11}) illustrating that even in the rational (polynomial) case, the situation in Theorem \ref{thm:main1} cannot be reduced to a situation covered by Theorem \ref{thm:main0}.

\begin{example}\label{ex:1}
Let $n=1$ and
\begin{equation}
R_\lambda(z,\bar z):=(1+|z|^2)^4-\lambda|z|^4=1+4|z|^2+(6-\lambda)|z|^4+4|z|^6+|z|^8.
\end{equation}
We observe that $R_\lambda(z,\bar )=1+Q_\lambda(z,\bar z)$ where $Q_\lambda(z,\bar z)$ is an SOS if and only if $\lambda\leq 6$. A straightforward calculation shows that
\begin{equation}
(1+|z|^2)R_\lambda(z,\bar z)=1+5|z|^2+(10-\lambda)|z|^4+(10-\lambda)|z|^6+5|z|^8+|z|^{10},
\end{equation}
which is of the form $1+\norm h\norm^2$ if and only $\lambda \leq 10$. Another calculation shows that
\begin{multline}\label{e:Rsq}
R_\lambda(z,\bar z)^2=1+1+8|z|^2+2(14-\lambda)|z|^4+8(7-\lambda)|z|^6+\left((6-\lambda)^2+34\right)|z|^8
\\+8(7-\lambda)|z|^{10}
+2(14-\lambda)|z^{12}+9|z|^{14}+|z|^{16},
\end{multline}
which is of the form $1+\norm f\norm^2$ if and only $\lambda \leq 7$. (According to \cite{JPDVar04} and \cite{JPD11}, there exists $k\geq 2$ such that $R_\lambda(z,\bar z)^k= 1+$ an SOS if and only if $\lambda<8$.) Thus, if we choose $\lambda=7$ and denote $R_7$ by $R$, then we have
\begin{equation}\label{e:Rprod}
(1+|z|^2)R(z,\bar z)=1+\norm h\norm^2,
\end{equation}
with $h=(h_1,\ldots, h_m)$ and $m=5$, but $R(z,\bar z)$ is not an SOS. However, if we square \eqref{e:Rprod} and use \eqref{e:Rsq}, then we obtain
\begin{equation}
(1+|z|^2)^2(1+\norm f\norm^2)=(1+\norm h\norm^2)^2,
\end{equation}
with $h=(h_1,\ldots, h_m)$ and $m=5$, $f=(f_1,\ldots,f_d)$ and $d=6$. Thus, this is an example where the conditions in Theorem \ref{thm:main1} are satisfied (with $a=b=2$, $c=1$, $e=d=6$), but which cannot be reduced to a situation covered by Theorem \ref{thm:main0}. Of course, in this example the lower bound for $m$ (which is 5) provided by Theorem \ref{thm:main1} (ii) (or the one provided by Huang's Lemma as above) is a poor estimate, as it reduces to $m\geq 1$ (which is a trivial bound) when $n=1$. We do note, however, that $m<d=e$ in contrast to the situation in Theorem \ref{thm:main0}.

\end{example}


\def\cprime{$'$}

\end{document}